\documentclass[A4paper,10pt]{article}
\usepackage{amsmath,amssymb,mathrsfs,pictex,graphicx}
\date{}

\def\scstyle{\scriptstyle}
\renewcommand{\mathcal}{\mathscr}
\def\cl#1{{\mathscr #1}}
\newcommand{\dlines}{\displaylines}
\def\ep{\varepsilon}

\def\paref#1{(\ref{#1})}

\newcommand{\field}[1]{\mathbb{#1}}
\newcommand{\R}{\field{R}}
\newcommand{\finedim}{\par\hfill$\blacksquare$\par\noindent\ignorespaces}

\newcommand{\RR}{\field{R}}

\newcommand{\bH}{\field{H}}

\let\pa=\partial
\def\acosh{\mathop{\rm acosh}}

\newcommand{\e}{{\rm e}}
\newcommand{\F}{{\mathscr{F}}}

\def\cl#1{{\mathcal #1}}

\def\E{{\rm E}}

\def\P{{\rm P}}

\def\bH{{\field H}}

\def\tfrac#1#2{{\textstyle\frac {#1}{#2}}}
\newtheorem{theorem}{Theorem}
\newtheorem{rema}[theorem]{Remark}
\numberwithin{equation}{section}

\newtheorem{assum}{Assumption}

\newtheorem{lemma}[theorem]{Lemma}
\newtheorem{lem}[theorem]{Lemma}
\newtheorem{prop}[theorem]{Proposition}

\begin{document}
\title{\Huge\sc Large Deviation asymptotics for the exit from a
domain of the bridge of a general Diffusion}

\author{Paolo Baldi, Lucia Caramellino, Maurizia Rossi\\
{\it Dipartimento di Matematica, Universit\`a
di Roma Tor Vergata, Italy}\\
} \date{}\maketitle

\begin{abstract}
We provide Large Deviation estimates for the  bridge of a
$d$-dimensional general diffusion process as the conditioning time tends to $0$ and apply these results to the evaluation of the asymptotics of its exit time probabilities. We are motivated by applications to numerical simulation, especially in connection with stochastic volatility models.
\end{abstract}

\smallskip


\section{Introduction}\label{intro}
In this paper we give Large Deviation (LD) estimates for the probability that
a $d$-dimensional pinned diffusion process
exits from a domain containing its endpoints, as the conditioning time goes to $0$.

This investigation is motivated by applications to numerical simulation in presence of a boundary, as explained in \cite{Bal:95} and \cite{MR1849251}, where a correction technique is developed requiring at each iteration of the Euler scheme to check whether the conditioned diffusion touches the boundary.

In the most common case, where the exit from a domain is actually the crossing of a barrier, people have resorted so far to the approximation of the conditioned diffusion with the one that has its coefficients frozen at one of the endpoints, so that the computation is reduced to exit probabilities for the Brownian bridge that are well known, also for curved boundaries (see \cite{lerche} and \cite{Bal:95}). In \cite{gobet} this technique is well developed producing also error estimates.

In recent years however the application to stochastic volatility models of the freezing approximation has prompted to further the investigation in order to obtain precise estimates and also to check the amount of error produced by the method of freezing.
As remarked in \cite{Bal:95} the exact computation of this probability is, in general, very difficult and it is suggested that as the time step in simulation is in general very small, one could be sufficiently happy with the equivalent of this probability as the conditioning time goes to $0$. The natural tool are therefore LD's or, better, sharp LD's, as developed in \cite{Bal:95} and \cite {MR1849251}.

Our first step is a pathwise Large Deviation Principle (LDP) for the bridge of a general diffusion whose proof turns out to be particularly simple. Our results are in some sense weaker than the existing literature, but apply to a general class of diffusions. See in particular \cite{bailleul}, \cite{hsu1-MR1027823} and \cite{inahama}.  As already mentioned this LD result is not a complete LDP for the conditioned diffusion but it is sufficient in order to derive the estimates for the probability of exit from a domain that are our final goal.

Considering applications to stochastic volatility models we must mention that some of them do not fall into the domain of application of our result. It is the case for instance of the Heston model. In the last section, on the other hand, we show how to apply our estimates to the Hull and White model.

In \S1 we introduce notations and recall known results needed in the sequel.
In \S2, using the well known equivalence between LDP and Laplace Principle we prove a Large Deviation estimate for a large class of conditioned diffusion processes. These LD estimates are sufficient in order to derive in \S4 LD estimate for the probability of the conditioned diffusion to exit a given, possibly unbounded, open set $D$. The main tool in this section is an extension of Varadhan's lemma to possibly discontinuous functionals. In \S5 we give an application to the Hull-White stochastic volatility model, pointing out in particular that the method of freezing the coefficients can lead to significantly wrong appreciations. The computations turn out to be particularly simple in this case as the model is then closely related to the geometry of the hyperbolic half-plane $\bH$.

\section{Notations}
In this paper $\xi$ denotes a $d$-dimensional diffusion process satisfying the
Stochastic Differential Equation (SDE) on an open set $S\subset\R^m$
\begin{equation}\label{SDE for X}
\begin{array}{l}
d\xi_s=b(\xi_s)\, ds+\sigma(\xi_s)\, dB_s\ ,\cr
\xi_0 = x
\end{array}
\end{equation}
for some locally Lipschitz coefficients  $b : S\to \R^d$ and $\sigma: S\to \R^{d\times m}$,
$B$ being a $m$-dimensional Brownian motion. We assume that $\xi$ has a transition density,
denoted  by $p$ from now on.
\begin{assum}\label{assum1} We assume that the SDE \paref{SDE for X} has a.s. infinite lifetime, i.e. that $\inf\{t,\xi_t\not\in S\}=+\infty$.
\end{assum}
For $t>0$ we denote $\xi^t$ the rescaled process, i.e. $\xi^t_s = \xi_{st}$ for $s\in [0,1]$; it is immediate that $\xi^t$
 satisfies the SDE
\begin{equation}\label{SDE rescaled X}
\begin{array}{l}
d\xi^t_s=t b(\xi^t_s)\, ds+\sqrt{t}\,\sigma(\xi^t_s)\, dB_s\ ,\cr
\xi^t_0=x\ ,
\end{array}
\end{equation}
for a possibly different Brownian motion $B$. For $y\in S$, the pinned process conditioned by $\xi_t=y$ is the time inhomogeneous Markov process $\widehat \xi=\widehat \xi_{x,y}$ whose transition
density is, for $0\le u\le v\le t$ and  $z_1, z_2\in \R^d$,
\begin{equation}\label{widehat-p}
\widehat p(u,v,z_1,z_2)=\frac {p(v-u,z_1,z_2)p(t-v,z_2,y)}{p(t-u,z_1,y)}\ \cdotp
\end{equation}
Let now $\widehat \xi^{t}=\widehat \xi^t_{x,y}$ be the rescaled pinned process given by
\begin{equation}\label{rescaled pinned process}
\widehat \xi^t_s = \widehat \xi_{ts}\ ,\qquad s\in [0,1]\ ,
\end{equation}
whose transition function is of course
\begin{equation}\label{widehat-p}
\widehat p_t(u,v,z_1,z_2)=\frac {p(t(v-u),z_1,z_2)p(t(1-v),z_2,y)}{p(t(1-u),z_1,y)}\ \cdotp
\end{equation}
Let us denote $\cl C_x ( [0,1], \R^d)$ (or for brevity $\cl C_x$) the space of continuous functions $\gamma:[0,1]\to \R^d$ with starting point $\gamma_0=x$. We denote $X_s$ the canonical application $X_s:\cl C_x\to\RR^d$ defined as $X_s(\gamma)=\gamma(s)$ and by $(\cl F_s)_s$ the canonical
filtration $\cl F_s=\sigma(X_u,u\le s)$.

We shall denote $\P^t_x$ the law on $\cl C_x$ of the rescaled process $\xi^t$ and
$\E^t_x$ the corresponding expectation and similarly we shall denote
$\widehat\P^t_{x,y}$ the law of  $\widehat \xi^t$. We have
\begin{equation}\label{radon-nikodym}
{ \frac {d\widehat\P^t_{x,y}}{d\P^t_x}}_{|_{\cl
F_s}}=\frac{p(t(1-s),X_s,y)}{p(t,x,y)}\ , \qquad 0\le s <1\ .
\end{equation}
Actually it is easy to check that under the probability
$$
\frac{p(t(1-s),X_s,y)}{p(t,x,y)}\, d\P^t_x\ ,
$$
the canonical process up to time $s$ is a Markov process
associated to the transition density $\widehat p_t$ in \paref{widehat-p}.

\paref{radon-nikodym} entails that the probability $\widehat\P^t_{x,y}$ is absolutely continuous with respect to $\widehat\P^t_{x}$ on $(\cl C_x,\cl F_s)$ for every $0\le s<1$ and with density given by the right-hand side of \paref{radon-nikodym}. Of course $\widehat\P^t_{x,y}$ cannot be absolutely continuous with respect to $\widehat\P^t_{x}$ on $\cl C_x$, as the event $\{X_1=y\}$ has probability $1$ under $\widehat\P^t_{x,y}$ and probability $0$ under $\widehat\P^t_{x}$.

Assume that $a(x)=\sigma(x) \sigma(x)^*:S\to \R^{d\times d}$ is invertible for every $x\in S$. Then thanks to the Freidlin-Wentzell theory (see \cite{FV:83}, \cite{az-stflour} or \cite{baldi-chaleyat} e.g.), $\xi^t$, as a $\cl C_x$-valued random variable, enjoys a LDP  with inverse speed $g(t) = t$ and rate function $I$ given, for $\gamma\in \cl C_x$, by
\begin{equation}\label{rate for rescaled}
I(\gamma) = \begin{cases}
\frac 12\int_0^1\langle a(\gamma_s)^{-1}\dot\gamma_s,\dot\gamma_s\rangle\, ds\ &\text{if}\ \gamma \ \text{is absolutely continuous}\ ,\\
+\infty\ &\text{otherwise}\ ,
\end{cases}
\end{equation}
where $\langle \cdot, \cdot \rangle$ denotes the standard inner product in $\R^d$ and $\dot\gamma$ stands for the derivative of $\gamma$. Remark that the drift $b$ does not appear in the expression of $I$. For $y\in\R^d$,
\begin{equation}\label{inf}
\inf_{\gamma\in \cl C_x,\gamma(1)=y}I(\gamma)=\frac 12\, d(x,y)^2\ ,
\end{equation}
where $d$ denotes the Riemannian distance on $S$ associated to the Riemannian metric
$(a(z)^{-1})_{z\in S}$ and also  that the infimum in \paref{inf} is actually
a minimum.

Remark that, thanks to the assumptions on $a$ (ellipticity and local Lipschitz continuity), the Riemannian metric and the Euclidean one on $S$ are
equivalent, in the sense that
$$
c|x-y|\le d(x,y)\le C |x-y|
$$
with constants $c,C$ that are uniform on compact sets. In particular the
paths $\{\gamma,I(\gamma)\le \rho\}$ are uniformly H\"older continuous both in the Euclidean and the $d$ distances.

Recall now Varadhan's estimate (\cite{Varadhan-variable-MR0208191}) which we shall need in the sequel
\begin{equation}\label{ve}
\lim_{t\to 0}t\log p(t,z_1,z_2)=-\frac 12\,d(z_1,z_2)^2\ ,
\end{equation}
this limit being moreover uniform on compact sets.
We shall need later the following version of Varadhan's Lemma (see \cite{PBBP} Theorem 9.2 e.g.).
\begin{theorem}\label{lem-varadhan}
Let $F\subset \cl C_x $ be a closed (resp. $G\subset \cl C_x$ open) set and $\Phi:\cl C_x\to \R$ a bounded continuous function.
Then, according to the previous notations,  we have
\begin{equation}\label{var-maj}
\limsup_{t\to 0}t\log\E^t_x\bigl[\e^{-\frac 1t\,\Phi}1_F\bigr]\le
-\inf_{\gamma\in F}[\Phi(\gamma)+I(\gamma)]\ ,
\end{equation}
$$\Big(resp.\
\liminf_{t\to 0}t\log\E^t_x\bigl[\e^{-\frac 1t\,\Phi}1_G\bigr]\ge
-\inf_{\gamma\in G}[\Phi(\gamma)+I(\gamma)]\Big)\ .
$$
\end{theorem}
\section{A Large Deviation result}
We shall obtain our results under the following additional assumption
\begin{assum}\label{assum2}
The open set $S$ endowed with the Riemannian distance $d$ is a complete manifold.
\end{assum}
Remark that Assumption \ref{assum2} is satisfied if also the solution of the SDE
$$
d\xi_s=\sigma(\xi_s)\, dB_s
$$
(i.e. the same equation as in \paref{SDE for X} but with the drift removed) has infinite
lifetime. This a consequence of Proposition 2.10 of \cite{az-stflour}. Unfortunately this
requirement is not satisfied in the Heston volatility model
\begin{equation}\label{heston}
\begin{array}{l}
dS_t=\mu\, dt+\sqrt{v_t}\, dB_1(t)\\
dv_t=(a-bv_t)\, dt+\rho\sqrt{v_t}\, dB_2(t)
\end{array}
\end{equation}
which lives in $S=\R^+\times \R$, as the boundary $\{0\}\times \R$ is at a finite
distance from every point of $S$ (see \cite{MR3168913} e.g.). In \S\ref{sec5} we shall see that, conversely, this assumption is satisfied in the Hull and White model.

Remark however that our results are known to be true also in situations in which neither Assumption \ref{assum2} nor Assumption \ref{assum1} hold. It is the case, for instance, of the CIR model (the process $(v_t)_t$ in the second equation of \paref{heston}). This is an immediate consequence of Theorem 2.1 of \cite{BC2002-MR1925452}.

Remark also that Assumption \ref{assum2} ensures that every ball of radius $R$ in the metric $d$ is relatively compact in the euclidean topology of $S$.

The main object of this section is the following LD estimate.
\begin{theorem}\label{main}
Let $\xi$ be the diffusion that is a solution of the $d$-dimensional
SDE \paref{SDE for X} and assume that the diffusion matrix $a=\sigma\sigma^*$ is invertible in
$S$ and that Assumptions \ref{assum1} and \ref{assum2} are satisfied.
For a fixed $y\in S$, let us consider the diffusion conditioned by $\xi_t=y$ for $t>0$ and
$(\widehat \xi_u^t)_{0\le u\le 1}$ the rescaled process $\widehat \xi^t_u=\widehat \xi_{ut}$.
Then, for every $0\le s<1$, $\widehat \xi^t$, as a $(\cl C_x,\cl F_s)$-valued random variable, satisfies
a LDP with respect to the rate function
\begin{equation}\label{rf0}
\widehat I(\gamma)=\begin{cases} I(\gamma)-I(\gamma_0)&\mbox{if }
\gamma_1=y\cr +\infty&\mbox{otherwise}
\end{cases}
\end{equation}
and inverse speed function $g(t)=t$, where $\gamma_0$ is the geodetic for the metric \paref{rate for rescaled} connecting $x$ to $y$, i.e.
\begin{align*}
&\limsup_{t\to0}t\log\P(\widehat \xi^t\in A_1)\le -\inf_{\gamma,\gamma\in A_1}\widehat I(\gamma)\cr
&\liminf_{t\to0}t\log\P(\widehat \xi^t\in A_2)\ge -\inf_{\gamma,\gamma\in A_2}\widehat I(\gamma)
\end{align*}
for every closed set $A_1\subset \cl C_x$ and every open set $A_2\subset \cl C_x$.
\end{theorem}
Remark that Theorem \ref{main} does not give a LDP on the
whole of $\cl F_1$, but that the rate function \paref{rf0} does not depend on $s$.
\medskip

\noindent \begin{proof}
It is well-known that the LDP is equivalent to the Laplace principle (see \cite{dupuis-ellis-MR1431744}, Theorem 1.2.3 e.g.), i.e. we must prove that for every bounded continuous $\cl F_s$-measurable functional $F: \cl C_x\to \R$,
\begin{equation}\label{Laplace}
\lim_{t\to 0}t\log\widehat \E^t_{x,y}\bigl[\e^{-\frac 1t\, F}\bigr]=-\inf_{\gamma\in \cl C_x} (F(\gamma)+\widehat I(\gamma))\ .
\end{equation}
$F$ being $\cl F_s$-measurable, we have by \paref{radon-nikodym}

\begin{equation}\label{eq1}
\log\widehat \E^t_{x,y}\bigl[\e^{-\frac 1t\, F}\bigr]=
-\log p(t,x,y)+\log\E^t_{x}\bigl[\e^{-\frac 1t\, F}p(t(1-s),X_s,y)\bigr]\ .
\end{equation}
Varadhan's estimate \paref{ve} immediately gives $\lim_{t\to 0}t\log p(t,x,y)=-\frac 12\,d(x,y)^2=-I(\gamma_0)$.
Let us turn our attention to the remaining term
$$
\log\E^t_{x}\bigl[\e^{-\frac 1t\, F}p(t(1-s),X_s,y)\bigr]=
\log\E^t_{x}\bigl[\e^{-\frac 1t\, (F-t\log p(t(1-s),X_s,y))}\bigr]\ .
$$
We shall now take advantage of the LD estimate enjoyed by the family of probabilities $(\P^t_x)_t$. If we were allowed to replace in the r.h.s. the quantity $t\log p(t(1-s),X_s,y)$ by $-\frac 1{2(1-s)}\, d(X_s,y)^2$, which is its limit as $t\to 0$ by \paref{ve}, we could apply
Varadhan's Lemma \ref{lem-varadhan} and the LDP that holds for
the unconditioned processes $(\P_x^t)_t$ and find that
\begin{equation}\label{heuristic}
\begin{array}{c}
\displaystyle\lim_{t\to 0}t \log
\E^t_{x}\bigl[\e^{-\frac 1t\, (F-t\log p(t(1-s),X_s,y))}\bigr]
=\displaystyle\lim_{t\to 0}t \log\E^t_{x}
\bigl[\e^{-\frac 1t\, (F+\frac 1{2(1-s)}\, d(X_s,y)^2)}\bigr]=\\
\displaystyle=-\inf_{\gamma\in \cl C_x}\Bigl(F(\gamma)+I(\gamma)+\frac 1{2(1-s)}\,d(\gamma_s,y)^2\Bigr)\ .
\end{array}
\end{equation}
In Lemma \ref{lb} we prove that
$$
\inf_{\gamma\in \cl C_x}\Bigl(F(\gamma)+I(\gamma)+\frac 1{2(1-s)}\,d(\gamma_s,y)^2\Bigr)=\inf_{\widehat\gamma, \widehat\gamma_1=y}(F(\widehat\gamma)+I(\widehat\gamma))
$$
so that \paref{heuristic} will allow to end the proof.
Let us prove \paref{heuristic} rigorously. Let
$$
k_s:=\inf_{\gamma\in \cl C_x}\Bigl(F(\gamma)+I(\gamma)+\frac 1{2(1-s)}\,d(\gamma_s,y)^2\Bigr)\ .
$$
Let, for $\ep>0$ fixed, $R$ be large enough so that
$$
\inf_{\gamma,\gamma_s\in \overline{B_R(x)}}\Bigl(F(\gamma)+I(\gamma)+\frac 1{2(1-s)}\,d(\gamma_s,y)^2\Bigr)\le k_s+\ep\ ,
$$
where $B_R(x)$ stands for the open ball for the distance $d$ centered at $x$ with radius $R$. Remark that, thanks to Assumption \ref{assum2}, $\overline{B_R(x)}\subset S$.
By Varadhan's estimate \paref{ve}, $\overline{B_R(x)}$ being compact, we have that for $t$ small and uniformly for $z\in \overline{B_R(x)}$,
$$
-\ep-\frac 1{2(1-s)}\,d(z,y)^2\le t\log p(t(1-s),z,y))\le\ep-\frac 1{2(1-s)}\,d(z,y)^2\ ,
$$
so that, by the LD lower bound enjoyed by $(\P^t_x)_t$,
$$
\dlines{
\liminf_{t\to 0}t \log\E^t_{x}
\bigl[\e^{-\frac 1t\, (F-t\log p(t(1-s),X_s,y))}\bigr] \ge\cr
\ge\liminf_{t\to 0}t \log
\E^t_{x}\bigl[\e^{-\frac 1t\, (F-t\log p(t(1-s),X_s,y))}
1_{\{X_s\in B_{R}(x)\}}\bigr]\ge\cr
\ge\liminf_{t\to 0}t \log
\E^t_{x}\bigl[\e^{-\frac 1t\, (F+\ep+\frac 1{2(1-s)}\,d(X_s,y)^2)}
1_{\{X_s\in B_{R}(x)\}}\bigr]\ge\cr
\ge-\inf_{\gamma,\gamma_s\in B_{R}(x)}
\Bigl(F(\gamma)+I(\gamma)+\frac 1{2(1-s)}\,d(\gamma_s,y)^2\Bigr)-\ep\ge\cr
\ge-\inf_{\gamma}\Bigl(F(\gamma)+I(\gamma)-\frac 1{2(1-s)}\,d(\gamma_s,y)^2\Bigr)-2\ep\ .
}
$$
Let us look for a converse inequality and decompose the quantity $\E^t_{x}\bigl[\e^{-\frac 1t\, (F-t\log p(t(1-s),X_s,y)}\bigr]$ into the sum of the two terms
\begin{equation}\label{2terms}
\begin{array}{c}
I_1(t)=\E^t_{x}\bigl[\e^{-\frac 1t\, (F-t\log p(t(1-s),X_s,y))}1_{\{X_s\in \overline{B_R(x)}\}}\bigr]\ ,\cr
I_2(t)=\E^t_{x}\bigl[\e^{-\frac 1t\, (F-t\log p(t(1-s),X_s,y))}1_{\{X_s\in {\overline{B_R(x)}}^{\,c}\}}\bigr]\ .
\end{array}
\end{equation}
By Varadhan's lemma \ref{lem-varadhan}, the set $\{\gamma_s\in \overline{B_R(x)}\}$ being closed, we have similarly
$$
\dlines{
I_1(t)=\limsup_{t\to 0}t \log\E^t_{x}\bigl[
\e^{-\frac 1t\, (F-t\log p(t(1-s),X_s,y))}1_{\{X_s\in \overline{B_R(x)}\}}\bigr]\le\cr
\le\limsup_{t\to 0}t \log\E^t_{x}\bigl[
\e^{-\frac 1t\, (F-\ep+\frac 1{2(1-s)}\,d(X_s,y)^2)}1_{\{X_s\in \overline{B_R(x)}\}}
\bigr]\le\cr
\le -\inf_{\gamma,\gamma_s\in \overline{B_R(x)}}
\Bigl(F(\gamma)+I(\gamma)+\frac 1{2(1-s)}\,d(\gamma_s,y)^2\Bigr)+\ep\le -k_s+2\ep\ .
}
$$
As for $I_2(t)$, let $R_0>0$ be large enough so that $y\in B_{R_0}(x)$ and let
$C=\sup_{u\le 1,z\in\pa B_{R_0}(x)} p(u,z,y)$. As
the density $p$ is a continuous function of $(t,x)$ as far as $z$
is away from $y$, we have $C<+\infty$.
as a function of $(t,x)$
By a simple application of the strong Markov property (see Proposition \ref{maximum} below), we have for $R>R_0$
$$
p(u,z,y)\le C\quad \mbox{for every }0\le u\le 1,z\in B_{R}(x)^c\ .
$$
Taking $R>R_0$, we have, denoting by $M$ a bound for $|F|$,
$$
\dlines{
I_2(t)\le\E^t_{x}\bigl[\e^{-\frac 1t\, (F-t\log p(t(1-s),X_s,y))}
1_{\{X_s\in B_R(x)^c\}}\bigr]\le
C\,\E^t_{x}\bigl[\e^{-\frac 1t\, F}1_{\{X_s\in B_R(x)^c\}}\bigr]\le\cr
\le C\,\e^{\frac Mt}\P^t_{x}(X_s\in B_R(x)^c)\ .\cr
}
$$
By the LDP satisfied by $X_s$ under $\P^t_{x}$ as $t\to 0$, we have for $t$ small
$$
\limsup_{t\to0}t\log\P^t_{x}(X_s\in B_R(x)^c)\le -\frac 1{2}\, R^2\ ,
$$
recall actually that, as $B_R(x)$ is the ball in the distance $d$, we have $I(\gamma)\ge \frac 12\, R^2$ for every path $\gamma$ such that $\gamma_s\in B_R(x)^c$. Thanks to Assumption \ref{assum2} we can choose $R$ as large as we want.
In particular we can have $R>M+k_s$. With this choice of $R$ we have
$$
\limsup_{t\to 0}t\log \E^t_{x}\bigl[\e^{-\frac 1t\, (F-t\log p(t(1-s),X_s,y))}1_{\{X_s\in B_R^c\}}\bigr]\le M-R<-k_s\ .
$$
Finally, putting together the estimates we have
$$
\limsup_{t\to 0}\E^t_{x}\bigl[\e^{-\frac 1t\, (F-t\log p(t(1-s),X_s,y))}\bigr]=\max(-k_s,M-R)=-k_s
$$
which finishes the proof.
\end{proof}
\finedim

\noindent Remark, as a consequence of Theorem \ref{main}, that the rate function $\widehat I$ of the rescaled bridge does not depend on the drift $b$ of \paref{SDE for X}.
\begin{lemma}\label {lb} For every continuous $\cl F_s$-measurable functional $F$ we have
\begin{equation}\label{hat-nohat}
\inf_\gamma\Bigl(I(\gamma)+F(\gamma)+\tfrac
1{2(1-s)}\,d(\gamma_s,y)^2\Bigr)=
\inf_{\widehat\gamma,\widehat\gamma_1=y}\bigl(I(\widehat\gamma)+
F(\widehat\gamma)\bigr)\ .
\end{equation}
\end{lemma}
\begin{proof}
The main point is the rather obvious inequality:
\begin{equation}\label{ineq-j}
\frac 1{1-s}\,d(\gamma_1,\gamma_s)^2\le\int_s^1\langle
a(\gamma_u)^{-1}\dot \gamma_u,\dot \gamma_u\rangle\, du
\end{equation}
which becomes an equality if $\gamma$ in the time interval $[s,1]$ is the
geodesic connecting $\gamma_s$ to $\gamma_1$.

Let $\widehat \gamma$ be such that $I(\widehat \gamma)<+\infty$ and $\widehat
\gamma_1=y$ and let $\gamma$ be the path that coincides with $\widehat
\gamma$ until time $s$ and is constant in $[s,1]$. We have $F(\widehat\gamma)=F(\gamma)$, as the two paths
$\gamma$ and $\widehat\gamma$ coincide up to time $s$ and also, of course,
$$
\int_0^s\langle a(\widehat
\gamma_u)^{-1}\dot {\widehat \gamma}_u,\dot {\widehat \gamma}_u\rangle\,
du=\int_0^1\langle a(
\gamma_u)^{-1}\dot {\gamma}_u,\dot { \gamma}_u\rangle\,
du\ .
$$
Therefore, thanks to \paref{ineq-j},
$$
\dlines{ I(\widehat \gamma)=\frac 12\,\int_0^s\langle a(\widehat
\gamma_u)^{-1}\dot {\widehat \gamma}_u,\dot {\widehat \gamma}_u\rangle\,
du+ \frac 12\,\int_s^1 \langle a(\widehat \gamma_u)^{-1}\dot {\widehat
\gamma}_u,\dot {\widehat \gamma}_u\rangle\, du
\ge\cr \ge\frac 12\,\int_0^1\langle
a(\gamma_u)^{-1}\dot \gamma_u,\dot \gamma_u\rangle\, du+\frac
1{2(1-s)}\,d(\gamma_s,y)^2 =I(\gamma)+\frac
1{2(1-s)}\,d(\gamma_s,y)^2\ .\cr }
$$
As $F(\gamma)=F(\widehat\gamma)$,
we have proved that, for every $\widehat\gamma$ such that $\widehat\gamma_1=y$, there exists an unconstrained path $\gamma$ such that
\begin{equation}\label{mino}
F(\gamma)+I(\gamma)+\tfrac
1{2(1-s)}\,d(\gamma_s,y)^2\le F(\gamma)+I(\widehat\gamma)\ .
\end{equation}
Conversely, for a given path $\gamma$ such that
$I(\gamma)<+\infty$, let $\widehat\gamma$ the path that coincides
with $\gamma$ up to time $s$ and then connects $\widehat\gamma_s=\gamma_s$
with the value
$\widehat\gamma_1=y$ with a geodesic in the time interval $[s,1]$.
For such a path $\widehat\gamma$ we have equality in \paref{ineq-j} and therefore
$$
\dlines{ I(\widehat \gamma)=\frac 12\,\int_0^s\langle a(\widehat
\gamma_u)^{-1}\dot {\widehat \gamma}_u,\dot {\widehat \gamma}_u\rangle\,
du+ \frac 12\,\int_s^1\langle a(\widehat \gamma_u)^{-1}\dot {\widehat
\gamma}_u,\dot {\widehat \gamma}_u\rangle\, du=\cr
=\frac
12\,\int_0^s\langle a(\gamma_u)^{-1}\dot {\gamma}_u,\dot
{\gamma}_u\rangle\, du+\frac 1{2(1-s)}\,d(\gamma_s,y)^2\le\cr \le \frac
12\,\int_0^1\langle a(\gamma_u)^{-1}\dot {\gamma}_u,\dot
{\gamma}_u\rangle\, du+\frac 1{2(1-s)}\,d(\gamma_s,y)^2=\cr
=I(\gamma) +  \frac 1{2(1-s)}\,d(\gamma_s,y)^2\ .}
$$
Therefore, again using the fact that $F(\gamma)=F(\widehat\gamma)$,
we see that in \paref{mino} also the $\ge$ sign holds, so that
the equality is established.
\end{proof}
\finedim
The following result, that was used in the proof of Theorem \ref{main}, is a
simple consequence of the strong Markov property.
\begin{prop}\label{maximum} Let $X$ be a $S$-valued continuous Markov process with transition density $p$ enjoying the Feller property. Let $D\subset S$ an open set and $y\in D$. Then if $z\in D^c$ we have
$$
p(t,z,y)=\E^z[1_{\{\rho<t\}}p(t-\rho,X_\tau,y)]\ ,
$$
where $\rho$ denotes the entrance time in $D$. In particular
$$
p(t,z,y)\le \sup_{s\le t,\zeta\in\pa D}p(s,\zeta,y)\ .
$$
\end{prop}
\section{Application: asymptotics of exit time probabilities}
Let $D\subset S$ be any open set (possibly unbounded)  having a $C^1$ boundary and such that $x,y\in D$.
Let us denote by $\tau$ the exit time of the rescaled bridge $\widehat \xi^t=\widehat \xi^t_{x,y}$ from $D$, i.e.
$$
\tau := \inf \lbrace 0<s<1 : \widehat \xi^t_s \notin D   \rbrace\ .
$$
We want to evaluate the LD asymptotics of the probability $\widehat \P^t_{x,y} \bigl( \tau < 1\bigr)$
as $t\to 0$, more precisely the limit
\begin{equation}\label{limite}
\lim_{t\to 0} t\log \widehat \P^t_{x,y} \bigl( \tau < 1\bigr)\ .
\end{equation}
If we had LD estimates up to time $1$ for the rescaled bridge,
the answer would immediately follow. However we shall see that the LD
estimates up to time $s < 1$ of Theorem \ref{main} are sufficient in order to obtain
\paref{limite}, even if we shall be confronted with some additional difficulties.
In what follows we are inspired by the proof of Lemma 4.2. in \cite{BC2002-MR1925452}.

Let $s<1$ and let us also denote $\tau:\cl C_x\to\R^+\cup\{+\infty\}$ the first time $u$ such that $\gamma(u)\not\in D$. Similarly $\tau_{\overline D}$ will be equal to the first time $u$ such that $\gamma(u)\not\in \overline D$. Remark that as the diffusion matrix $a$ is assumed to be positive definite the exit times $\tau$ and $\tau_{\overline D}$ are $\P^t_x$-a.s.-equal  for every $t>0$. Thanks to Theorem \ref{main} we have
\begin{equation}\label{inf-s}
\lim_{t\to 0}t\log\widehat \P^t_{x,y} \bigl( \tau < s\bigr)=\lim_{t\to 0}t\log\widehat \P^t_{x,y} \bigl( \tau_{\overline D} < s\bigr)\ge-\inf_{\gamma,\tau_{\overline D}(\gamma)<s}\widehat I(\gamma)\ .
\end{equation}
Actually the functional $\tau_{\overline D}:\cl C_x\to\R^+\cup\{+\infty\}$ is upper semi-continuous and the set $\{\gamma;\tau_{\overline D}(\gamma) < s\}\subset \cl C_x$ is open. Of course if $\tau(\gamma_0)<1$, i.e. the geodesic joining $x$ to $y$ exits from $D$ then the infimum at the right-hand side in \paref{inf-s} is equal to $0$ and the probability of exit goes to $1$. Otherwise we have
$$
\inf_{\gamma,\tau(\gamma)<s}\widehat I(\gamma)=\inf_{z\in\pa D}\,\,\inf_{u<s}\,\,
\inf_{\gamma,\gamma(u)=z}\widehat I(\gamma)\ .
$$
It is clear that
\begin{equation}\label{infimum}
\inf_{\gamma,\gamma(u)=z}\widehat I(\gamma) =\frac 12\Bigl(\frac 1u\, d(x,z)^2+
\frac 1{1-u}\,d(z,y)^2-d(x,y)^2\Bigr)\ .
\end{equation}
An elementary computation gives that the infimum over $u\in]0,1[$ of the quantity above is attained at
$$
\overline u=\frac {d(x,z)}{d(x,z)+d(z,y)}
$$
and that the value of the minimum for $u\in[0,1]$ of the quantity in \paref{infimum}
is equal to
$$
\frac 12\,\Bigl(\bigl(d(x,z)+d(z,y)\bigr)^2-d(x,y)^2\Bigr)\ .
$$
Therefore
\begin{equation}\label{lb}
\lim_{t\to 0}t\log\widehat \P^t_{x,y} \bigl( \tau < s\bigr)\ge-\inf_{z\in\pa D}
\frac 12\,\Bigl(\bigl(d(x,z)+d(z,y)\bigr)^2-d(x,y)^2\Bigr)
\end{equation}
{\it for every $s\ge \overline u$}, i.e. the asymptotics of $t\log\widehat \P^t_{x,y} \bigl( \tau < s\bigr)$ does not depend on $s$ as soon as $s$ is close enough to $1$. This gives immediately
$$
\liminf_{t\to 0}t\log\widehat \P^t_{x,y} \bigl( \tau < 1\bigr)\ge
-\inf_{z\in\pa D}
\frac 12\,\Bigl(\bigl(d(x,z)+d(z,y)\bigr)^2-d(x,y)^2\Bigr)\ .
$$
In order to prove the next statement we just need to take care of the converse inequality.
\begin{prop}\label{upper}
Let $D$ be an open set having a $C^1$ boundary. Then
\begin{equation}
\lim_{t\to 0} t\log \widehat \P^t_{x,y} \bigl( \tau < 1\bigr)=
-\inf_{z\in\pa D}
\frac 12\,\Bigl(\bigl(d(x,z)+d(z,y)\bigr)^2-d(x,y)^2\Bigr)\ .
\end{equation}
\end{prop}
\begin{proof}
Recalling that $\frac {d\widehat\P^t_{x,y}}{d\P^t_x}_{|_{\cl
F_{s}}}\!\!\!\!=\frac{p(t(1-s),X_{s},y)}{p(t,x,y)}$, which is a $( \F_s )_{0\le s <1}$- martingale and conditioning w.r.t. $\F_{\tau \wedge s}$, we have
$$\displaylines{
\widehat \P^t_{x,y} ( \tau < s ) =
\frac{1}{p(t,x,y)} \E^t_{x} \bigl[ p(t (1-s), X_{s}, y) 1_{\{\tau < s\}}     \bigr] =\cr
=\frac1{p(t,x,y)}\, \E^t_{x}\bigl[1_{\{\tau < s\}}    \E^t_{x}   \bigl[ p(t (1-s), X_{s}, y)\, |\, \F_{\tau \wedge s}  \bigr]  \bigr]=\cr
= \frac1{p(t,x,y)}\, \E^t_{x}\bigl[1_{\{\tau < s\}}     p(t (1- \tau\wedge s), X_{\tau \wedge s}, y) \bigr]= \cr
 = \frac{1}{p(t,x,y)}\E^t_{x}\bigl[p(t(1-\tau), X_{\tau}, y)1_{\{\tau < s\}}      \bigr]
}$$
and, taking the limit as $s \uparrow 1$ with Beppo Levi's theorem,
$$
 \widehat \P^t_{x,y} ( \tau < 1)  =
\frac{1}{p(t,x,y)}\E^t_{x}\bigl[p(t(1-\tau), X_{\tau}, y)1_{\{\tau < 1\}}      \bigr]\ ,
$$
so that by \paref{ve}
\begin{equation}
\limsup_{t\to 0} t\log \widehat \P^t_{x,y} ( \tau < 1)
\le\frac{1}{2}d(x,y)^2 + \limsup_{t\to 0} t\log \E^t_{x}\bigl[p(t(1-\tau), X_{\tau}, y)1_{\{\tau < 1\}}      \bigr]\ .
\end{equation}
Our aim now is to apply Varadhan's lemma to the $\limsup$ at the right-hand side and to apply Varadhan's estimate \paref{ve} to the term $p(t(1-\tau), X_{\tau}, y)$. The difficulty coming from the lack of continuity of the functional $\tau$ is overcome in the next Proposition \ref{propF} but, in order to apply \paref{ve} we must reduce the problem to the case of a bounded set $D$. Let $B_R(y)\subset S$ be the open ball of radius $R>0$ centered at $y$
with respect to the metric $d$ and let $\tau_R$ be the exit time from $D\cap B_R(y)$.
$\tau_R$ is a lower semi-continuous functional on $\cl C_x$, $D\cap B_R(y)$ being an open set.
Of course $\tau_R \le \tau$ and therefore
\begin{equation}\label{majoR}
\widehat \P^t_{x,y} ( \tau < 1) \le  \widehat \P^t_{x,y} ( \tau_R < 1)\ .
\end{equation}
We have
$$\displaylines{
\limsup_{t\to 0} t\log \widehat \P^t_{x,y} ( \tau_R < 1)=\frac{1}{2}d(x,y)^2 +
\limsup_{t\to 0} t\log \E^t_{x}\bigl[p(t(1-\tau_R), X_{\tau_R}, y)1_{\{\tau_R < 1\}}\bigr]\ .
}$$
Thanks to \paref{ve}, $\overline {D\cap B_R(y)}$ being compact, we have for $t$ small
$$
\E^t_{x}\bigl[p(t(1-\tau_R), X_{\tau_R}, y)1_{\{\tau_R < 1\}}\bigr]\le
\E^t_{x}\bigl[\e^{-\frac{1}{t}\left(\frac 1{2(1-\tau_R)}\,d(X_{\tau_R},y)^2 - \ep\right)}
1_{\{\tau_R \le  1\}}\bigr]\ .
$$
Proposition \ref{propF} below states that the upper bound of Varadhan's lemma also holds for the functional inside the expectation at the right-hand side above. Hence applying Proposition \ref{propF} to the functional $F(\gamma)=
\frac 1{2(1-\tau_R)}\,d(\gamma_{\tau_R},y)^2 - \ep$
we obtain
$$
\lim_{t\to 0} t\log \E^t_{x}\bigl[p(t(1-\tau_R), X_{\tau_R}, y)1_{\{\tau_R \le 1\}}\bigr]\le
-\underbrace{\inf_{\gamma,\tau_R(\gamma)\le 1}
\Big(I(\gamma)+\frac 1{2(1-\tau_R)}\,d(\gamma_{\tau_R},y)^2 \Big)}_{:= M_R} + \ep\ .
$$
Let
$$
M:=\inf_{\tau(\gamma)\le 1}
\Big(I(\gamma)+\frac 1{2(1-\tau)}\,d(\gamma_{\tau},y)^2 \Big)\ ,
$$
and let us prove that $R$ can be chosen large enough so that $M_R= M$.
For every $\ep >0$, there exists a path $\gamma_\ep$, such that $\tau(\gamma_\ep)\le 1$ and
$$
I(\gamma_\ep)+\frac 1{2(1-\tau(\gamma_\ep))}\,d(\gamma_\ep({\tau}),y)^2 < M + \ep\ .
$$
Of course $\gamma_\ep([0,\tau_{R_\ep}])\subset B_{R_\ep}(y)$ for some $R_\ep>0$.
Let us consider $M_R$ for $R\ge R_\ep$.
Fix $\gamma$ such that $\tau_R(\gamma)\le 1$. If $\gamma(\tau_R)=\gamma(\tau) \in \partial D$, then
$$
M \le
I(\gamma)+\frac 1{2(1-\tau_R(\gamma))}\,d(\gamma(\tau_R),y)^2\ .
$$
Conversely, if $\gamma(\tau_R) \notin \partial D$, then necessarily $\gamma(\tau_R) \in \partial B_R(y)$ so that
$$
I(\gamma)+\frac 1{2(1-\tau_R(\gamma))}\,d(\gamma(\tau_R),y)^2 \ge\frac{1}2\, d(\pa B_R(y),y)^2=\frac{1}2\, R^2\ .
$$
Putting things together we find
$$\displaylines{
M_R \ge \inf \Bigl(\inf_{\gamma_{\tau_R}\in \partial D} I(\gamma)+\frac 1{2(1-\tau_R(\gamma))}
\,d(\gamma_{\tau_R},y)^2,   \frac{1}2 R^2\Bigr)\ge
\inf \Bigl(M, \frac{1}2 R^2\Bigr)\ .
}$$
For $R$ large enough we have therefore
$$
M_R = M\ ,
$$
and
$$
\lim_{t\to 0} t\log \E^t_{x}\bigl[p(t(1-\tau_R), X_{\tau_R}, y)1_{\{\tau_R \le 1\}}\bigr]\le - M + \ep\ .
$$
which together with \paref{majoR} allows to conclude.
\end{proof}
\finedim
\begin{lem}\label{lemlem}
Let $D\subset\RR^m$ be an open set having a $C^1$ boundary.
Let us denote by $\tau:\cl C_x\to\R^+\cup\{+\infty\}$ the
exit time from $D$. Let $\ep>0$, $\rho>0$. Then for every $\gamma\in \{I\le \rho\}\cap \{\tau\le1\}$
there exists $\widetilde \gamma\in \{I<+\infty\}\cap \{\tau\le1\}$ such that
$\Vert\gamma-\widetilde \gamma\Vert_\infty<\ep$ and which is a continuity point for $\tau$.
\end{lem}
\begin{proof}
It is easy to check that $\tau$ is a lower semicontinuous functional on $\cl C_x$,
whereas $\tau_{\overline D}$, the exit time from the closure $\overline D$ is upper
semicontinuous. $\tau$ is therefore continuous at $\gamma$ if and only if
$\tau(\gamma)=\tau_{\overline D}(\gamma)$.

Let us first assume that $\gamma\in \{I\le \rho\}\cap \{\tau<1\}$ (i.e. that it reaches $\pa D$ strictly before time $1$) and let $v$ a unitary vector that is
orthogonal to $\pa D$ at $\gamma(\tau)$ and pointing outward of $D$.

Let $\delta>0$. We define $\widetilde \gamma$ in the following way: $\widetilde \gamma$ is equal to
$\gamma$ until time $\tau(\gamma)$, then follows the direction $v$ at speed $1$ during a time
$\delta$, then goes back following the direction $-v$ during a time $\delta$. At this point
$\widetilde \gamma$ is at the position $\gamma(\tau)$. From now on it will follow the path $\gamma$ delayed by time $2\delta$. To be formal
$$
\widetilde \gamma(t)=\begin{cases}
\gamma(t)&\mbox{if }t\le \tau(\gamma)\cr
\gamma(\tau)+(t-\tau(\gamma))v&\mbox{if }\tau(\gamma)\le t\le \tau(\gamma)+\delta\cr
\gamma(\tau)+\delta v-(t-\tau(\gamma)-\delta)v&\mbox{if }\tau(\gamma)+\delta\le t\le \tau(\gamma)+2\delta\cr
\gamma(t-2\delta)&\mbox{if }t\ge \tau(\gamma)+2\delta\ .\cr
\end{cases}
$$
It is clear that, for every $\delta>0$, $\tau(\widetilde \gamma)=\tau_{\overline D}(\widetilde \gamma)$, as following the normal $v$ $\widetilde \gamma$ exits of $\overline D$ immediately ``after'' time $\tau$. Thus $\widetilde \gamma$ is a continuity point of $\tau$.
Moreover, exploiting the fact that $\gamma$ is H\"older continuous with some constant $K$,
it is easy to see that
$$
\Vert\gamma-\widetilde \gamma\Vert_\infty\le \delta+K\sqrt{2\delta}
$$
and this, for every $\ep>0$, can be made $<\ep$, $\delta$ being arbitrary.

If $\gamma\in \{I\le \rho\}\cap \{\tau=1\}$ the construction above does not apply but we can define, for $\delta>0$, $\widetilde\gamma_t=\gamma_{\frac{t}{1-\delta}}$ for $0\leq t\leq 1-\delta$ and $\widetilde\gamma_t=\gamma_1+(t-(1-\delta))v$ for $1-\delta\leq t\leq 1$ (i.e. $\widetilde\gamma$ is slightly accelerated with respect to $\gamma$ and uses the extra time in order to enter immediately into ${\overline D}^c$). It is easy to check that $\widetilde\gamma$ satisfies the requirements of the statement.
\end{proof}
\finedim
\begin{prop}\label{propF}
Let $F:\cl C_x\to\RR$ be a measurable function which is  bounded below ($F\ge -M$ for some
$M\ge 0$). 
Let us assume that for every
$\ep>0$, $\rho>0$ and for every $\gamma\in \{I\le \rho\}\cap \{\tau\le1\}$
there exists $\widetilde \gamma\in \{I<+\infty\}\cap \{\tau\le1\}$ such that
$\Vert\gamma-\widetilde \gamma\Vert_\infty<\ep$ and which is a continuity point of $F$. Then we have
\begin{equation}\label{ps-varadhan}
\limsup_{t\to+\infty}t\log\E^t_x\bigl[\e^{-F/t}1_{\{\tau\le1\}}\bigr]\le
-\inf_{\gamma,\tau(\gamma)\le 1}(F+I)\ .
\end{equation}
\end{prop}
\begin{proof} Let us fix $\ep>0$.
Let us denote by $G$ the set of continuity points $\gamma$ of $F$ such that $I(\gamma)<+\infty$ and $\tau(\gamma) \le 1$.

For every $\gamma\in G$, let $\delta=\delta(\gamma)$ be such that if
$\Vert \zeta-\gamma\Vert_\infty<\delta$, then $|F(\zeta)-F(\gamma)|<\ep$ and
$I(\zeta)\ge I(\gamma)-\ep$.
Let us denote $S_{\gamma,\delta}$ the open ball centered at
$\gamma$ and with radius $\delta$ in the uniform
topology.

Let us fix $\rho>0$. By hypothesis
the family $(S_{\gamma,\delta(\gamma)})_{\gamma\in G}$ is an open cover
of $\{I\le \rho\}\cap \{ \tau\le 1 \}$ which is a compact set.
Let $(S_{\gamma_i,\delta(\gamma_i)})_{i=1,\dots,m}$
a finite subcover. Then
$$
\E^t_x\bigl[\e^{-F/t}1_{\{\tau\le1\}}\bigr]\le
\sum_{i=1}^m \int_{S_{\gamma_i,\delta(\gamma_i)}}\e^{-F/t}1_{\{\tau\le1\}}\, d\P^t_x+
\int_\Gamma \e^{-F/t}\, d\P^t_x
$$
where $\Gamma=\{I > \rho \}\cap\{\tau\le1\}\setminus\bigcup_{i}
S_{\gamma_i,\delta(\gamma_i)}$.
Therefore for $i=1,\dots m$
$$
t\log \int_{S_{\gamma_i,\delta(\gamma_i)}} \e^{-F/t}\, d\P^t_x
\le t\log \Bigl(\e^{-(F(\gamma_i)-\ep)/t} \P^t_x(S_{\gamma_i,\delta(\gamma_i)})\Bigr)
$$
so that there exists $t_0>0$ such that, for $t\le t_0$,
$$
t\log \int_{S_{\gamma_i,\delta(\gamma_i)}} \e^{-F/t}\, d\P^t_x\le -F(\gamma_i)+\ep-I(\gamma_i)+2\ep
$$
and such that
$$
t\log \int_\Gamma \e^{-F/t}\, d\P^t_x\le t\log \bigl(\e^{M/t}
\P^t_x(\Gamma)\bigr)\le M-\rho+\ep\ .
$$
Therefore, for $t\le t_0$,
$$
\dlines{
t\log \E^t_x\bigl[\e^{-F/t}1_{\{\tau\le1\}}\bigr]\le
 \max\bigl(-F(\gamma_i)-I(\gamma_i)+3\ep,i=1,\dots,m,M-\rho\bigr)=\cr
=-\min
\bigl(F(\gamma_i)+I(\gamma_i)-3\ep,i=1,\dots,m,-M+\rho\bigr)\cr
}
$$
and $\rho,\ep$ being arbitrary, we find \paref{ps-varadhan}.
\end{proof}
\finedim

\begin{rema}\rm In \cite{bailleul} it is proved a LD estimate for the bridge of a hypoelliptic diffusion process. It is however required that the process takes its values on a compact manifold. Remark that Theorem \ref{main} remains true under this weaker assumption, just replacing the Riemannian distance with the associated sub-Riemannian distance (see \cite{bailleul} e.g. for details) and Varadhan's estimate \paref{ve} with its hypoelliptic counterpart (\cite{leandre-min}, \cite{leandre-maj}). As for the estimates of exit from a domain the main missing link is Lemma \ref{lemlem}: it is not obvious that given a point $x\in\pa D$ there exists a path of finite energy for which the exit time from $D$ and the one from $\overline D$ coincide.

Recall that a point $x\in\pa D$ is said to be {\it characteristic} if the subspace linearly generated by the random fields in the diffusion coefficient at $x$ is tangent to $\pa D$. We have that Lemma \ref{lemlem}, and therefore the upper bound, still holds in the hypoelliptic case if the boundary $\pa D$ does not have characteristic points and under this additional assumption also the argument concerning the lower bound remains valid. We choose not to address this issue as the applications that are the object of our interest are not in this direction.
\end{rema}
\section{Application: the Hull-White stochastic volatility model}\label{sec5}
As mentioned in \S \ref{intro}, the interest of estimates of the kind of Proposition \ref{upper} concerns mainly the application to simulation.
Let us consider the following model
\begin{equation}\label{HW-general}
\begin{array}{l}
\displaystyle d\xi_t=\Bigl(b-\frac {v_t^2}2\Bigr)\, dt+v_t\rho\, dB_1(t)+v_t\sqrt{1-\rho^2}\, dB_2(t)\\
\displaystyle dv_t=\mu v_t\, dt+\sigma v_t\, d B_1(t)\ ,
\end{array}
\end{equation}
where $-1\le \rho\le 1$, $\sigma>0$ and $B=(B_1,B_2)$ is a $2$-dimensional Brownian motion.
This is actually the Hull and White's model of stochastic volatility.
Remark that the process $(\xi,v)$ takes its values in $\RR\times\RR^+$.

We first investigate the simpler model with $\rho=0$ and $\sigma=1$, i.e.
\begin{equation}\label{HW-simple}
\begin{array}{l}
d\xi_t=\Bigl(b-\frac {v_t^2}2\Bigr)\, dt+v_t\, dB_1(t)\\
dv_t=\mu v_t\, dt+v_t\, d B_2(t)\ .
\end{array}
\end{equation}
The corresponding metric tensor is
\begin{equation}\label{metric a}
a^{-1}(\xi,v)=\begin{pmatrix} \frac 1{v^2}&0\cr 0&\frac 1{v^2}\end{pmatrix}
\end{equation}
and we recognize the metric structure of the Poincar\'e half-space $\bH$.
Recall that $\bH=\RR\times \RR^+$ with the Riemannian distance (see \cite{abate} e.g. or wikipedia)
\begin{equation}\label{HyperbolicDist}
d((x_1,y_1),(x_2,y_2))=\acosh\Bigl(1+\frac{(x_2-x_1)^2+(y_2-y_1)^2}{2y_1 y_2}\Bigr)\ ,
\end{equation}
whereas the geodesics are pieces of circles whose center lies on the $y=0$ axis.
Remark that distances become very large near the axis $y=0$,
actually
\begin{equation}\label{HyperbolicLim}
\lim_{y\to0}d((x_1,y_1),(x_2,y))=+\infty\ ,
\end{equation}
so that the boundary is at an infinite distance
from every point of $\bH$. It is easy to conclude that Assumption \ref{assum2} is satisfied and that Proposition \ref{upper} holds in this situation.
\medskip

\begin{figure}[h]
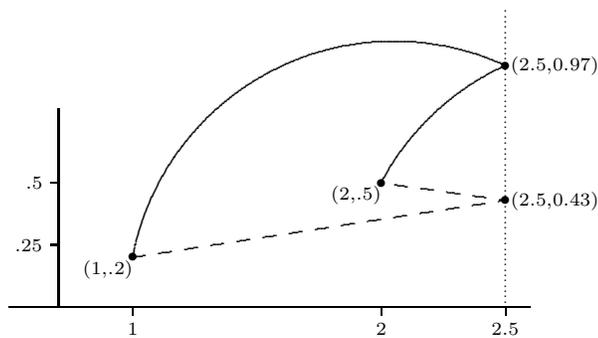

$$
\hbox{\hss
\beginpicture
\setcoordinatesystem units <1.3truein,1.3truein>
\setplotarea x from 0.5 to 2.6, y from 0 to .8
\axis bottom ticks short withvalues {$\scstyle 1$} {$\scstyle 2$} {$\scstyle 2.5$} / at 1 2 2.5 / /
\axis left shiftedto x=0.7 ticks short withvalues $\scstyle .25$ $\scstyle .5$ / at .25 .5 / /
\circulararc -103.93041 degrees from 1 .2  center at 2.0525 0
\circulararc -37.489499 degrees from 2 .5  center at 2.9475 0
\put {$\scstyle(1,.2)$} at .9 .15
\put {$\scstyle(2,.5)$} at 1.9 .45
\put {$\scstyle(2.5,0.97)$} [l] <2pt,0pt> at 2.5 0.9733961
\put {$\scstyle(2.5,0.43)$} [l] <2pt,0pt> at 2.5 0.43
\multiput {$\scstyle\bullet$} at 1 .2  2 0.5  2.5 0.9733961 2.5 0.43 /
\setdots <2pt>
\plot 2.5 0  2.5 1.2 /
\setdashes
\plot 1 .2  2.5 0.43  2 0.5 /
\endpicture\hss}
$$
\vglue2pt
\caption{Graph of the true minimizer (solid) in comparison with the one
provided by the approximation with the frozen coefficients (dashes). The dotted line denotes the barrier. Here ordinates are volatility whereas the logarithm of the price are in the abscissas. \label{HW1}}
\end{figure}
Figure \ref{HW1} shows the path minimizing the functional for the conditioned diffusion
with endpoints $(1,.2)$ and $(2,.5)$ in comparison with the one that would be given
by the method of freezing the coefficients (as the metric tensor is always a multiple of the identity, the minimizer does not depend on the point at which the coefficients are frozen, whereas the value of the rate function does). It appears that the two behaviors are quite different.
Not surprisingly also the values of the rate function produced by the two models are very different:
freezing the coefficients at $(2,.5)$ would give the value $6.0$, against the value with
the true model $3.8$ whereas freezing at $(1,.2)$ or at some point between would give an even larger result. Of course this discrepancy would be more large, the more the coordinate
$y$ of the volatility approaches $0$. Therefore freezing the coefficients would give for the probability $p_t$ of touching the barrier the equivalent as $t\to 0$ (in the logarithmic sense)
$$
p_t\approx \e^{-6/t}
$$
much different from the true equivalent $p_t\approx \e^{-3.8/t}$. Of course this discrepancy between the true equivalent and the one produced by the method of freezing will be important when the volatility is low and perhaps acceptable when it is high.

In Figure \ref{HW2} we considered endpoints nearer to the barrier, 
a situation more realistic during a simulation. The true value of the rate 
function turns out to be equal to $0.119$, whereas the coefficients frozen at $(2.48,.12)$ give the
approximation $0.083$. the discrepancy between the two values, assuming a time step 
$t=\frac 1{20}$ would give for the probability of exit the values
$$
\approx\e^{-.119\cdot 20}=0.092\qquad\mbox{and}\qquad\approx\e^{-.083\cdot 20}=0.19
$$
which is certainly not satisfactory. In this case freezing at the midpoint would give a
good approximation ($0.090$) but, as indicated by the situation of Figure \ref{HW1}, 
this is not a general rule.
\begin{figure}[h]
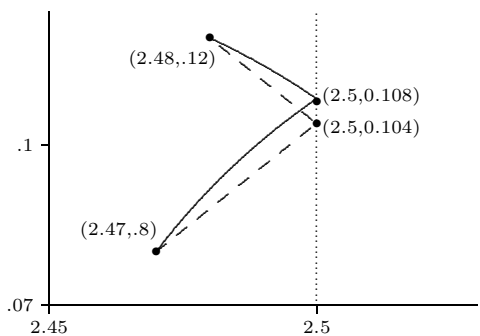

$$
\hbox{\hss
\beginpicture
\setcoordinatesystem units <28truein,28truein>
\setplotarea x from 2.45 to 2.53, y from 0.07 to .125
\axis bottom ticks short withvalues   {$\scstyle 2.45$} 
{$\scstyle 2.5$} / at  2.45 2.5 / /
\axis left shiftedto x=2.45 ticks short withvalues $\scstyle .07$ $\scstyle .1$ / at .07 .1 / /
\circulararc -18.073652 degrees from 2.47 .08  center at 2.575  0
\circulararc  -9.9988358 degrees from 2.48 .12  center at 2.425  0
\put {$\scstyle(2.47,.8)$} at 2.463 .084
\put {$\scstyle(2.48,.12)$} at 2.473 .116
\put {$\scstyle(2.5,0.108)$} [l] <2pt,0pt> at 2.5 0.109
\put {$\scstyle(2.5,0.104)$} [l] <2pt,0pt> at 2.5 0.103
\multiput {$\scstyle\bullet$} at 2.47 .08  2.48 .12  2.5 0.108 2.5 0.104 /
\setdots <2pt>
\plot 2.5 0.07  2.5 .125 /
\setdashes
\plot 2.47 .08  2.5 0.104  2.48 .12 /
\endpicture\hss}
$$
\vglue2pt
\caption{Same as in Figure \ref{HW1} but with values closer to the barrier.\label{HW2}}
\end{figure}

Here we computed numerically the  the point at which the path that minimizes the rate function reaches the boundary.

Let us go back to the general model \paref{HW-general}: the corresponding metric tensor is
\begin{equation}\label{metric tilde a}
\widetilde a^{-1}(\xi,v)=\frac{1}{v^2 \sigma^2 (1-\rho^2)}\begin{pmatrix} \sigma^2 & - \sigma \rho
\cr -\sigma \rho & 1\end{pmatrix}\ .
\end{equation}
Denoting, for simplicity, $\overline\rho=\sqrt{1-\rho^2}$, the linear transformation $\bH\to\bH$
\begin{equation}
A=\begin{pmatrix} \frac{1}{\overline\rho} & - \frac{\rho}{\sigma \overline\rho}
\cr 0 & \frac{1}{\sigma} \end{pmatrix}\
\end{equation}
turns the Riemannian metric \paref{metric tilde a}  on $\bH$  into the
Poincar\'e metric. It is then immediate that the Riemannian distance $\widetilde d$ associated to
\paref{metric tilde a}
is
\begin{equation}
\widetilde d((x_1,y_1) , (x_2,y_2))=
d((\overline\rho x_1 + \rho y_1, y_1) ,
(\overline\rho x_2 + \rho y_2, y_2))\ ,
\end{equation}
where $d$ is defined as in \paref{HyperbolicDist}. The geodesic for $d$ corresponding to
a piece of the circle centered in $(\alpha,0)$ with radius $R$ is changed into the geodesic
for $\widetilde d$, i.e. a
piece of the
 ellipse satisfying the equation
$$
(\overline\rho x + \rho y -\alpha)^2 + \sigma^2 y^2  = R^2\ .
$$
In practice in order to compute the equivalent of touching a barrier $x=x_0$ for the general Hull-White model \paref{HW-general} for the diffusion conditioned between the endpoints $z_1=(x_1,y_1)$ and $z_2=(x_2,y_2)$ the simplest way is to switch to the consideration of the simpler model \paref{HW-simple} conditioning between the endpoints $z_1'=Az_1$ and
$z_2'=Az_2$ and to compute, as indicated above, the equivalent of touching the straight line $x+\frac {\rho}{\overline\rho}\, y=\frac 1{\overline\rho}\, x_0$,which is the image of $x=x_0$ through $A$.

In this way it is possible to take advantage of the well known properties and
computations of the hyperbolic half plane $\bH$.

\def\cprime{$'$} \def\cprime{$'$}
\providecommand{\bysame}{\leavevmode\hbox to3em{\hrulefill}\thinspace}
\providecommand{\MR}{\relax\ifhmode\unskip\space\fi MR }
\providecommand{\MRhref}[2]{%
  \href{http://www.ams.org/mathscinet-getitem?mr=#1}{#2}
}
\providecommand{\href}[2]{#2}

\end{document}